\documentclass{amsart}

\usepackage{amsmath,amsthm,amssymb,latexsym,mathrsfs,enumerate,enumitem,multicol}
\usepackage[misc,geometry]{ifsym}
\usepackage[numbers]{natbib}

\usepackage[colorlinks=true,linkcolor=blue,citecolor=blue,urlcolor=blue]{hyperref}
\usepackage[utf8]{inputenc}
\usepackage[T1]{fontenc}

\numberwithin{equation}{section}
\theoremstyle{thmit} 
\newtheorem{theorem}{Theorem}[section]
\newtheorem{lemma}[theorem]{Lemma}
\newtheorem{corollary}[theorem]{Corollary}
\newtheorem{proposition}[theorem]{Proposition}

\newtheorem{example}[theorem]{Example}
\newtheorem{remark}[theorem]{Remark}

\newcommand\ces{\mathsf{C}}
\newcommand\ices{\mathsf{C}^{-1}}
\newcommand\cn{\mathbb{C}}
\newcommand\rn{\mathbb{R}}
\newcommand\cnn{\mathbb{C}^\mathbb{N}}
\newcommand\nn{\mathbb{N}}

\newcommand\kotinf{\lambda_1(A)}
\newcommand\ginf{G_\infty}
\newcommand\pssi{\Lambda^1_\infty(\alpha)}

\newcommand\ptsp{\sigma_{pt}(\mathsf{C}; \kozero)}
\newcommand\stsp{\sigma^*(\mathsf{C}; \kozero)}
\newcommand\spec{\sigma(\mathsf{C}; \kozero)}
\newcommand\clo{\mathcal{L}}
\newcommand\ind{\operatorname{ind}}
\newcommand\proj{\operatorname{proj}}
\newcommand\koinf{k_\infty(V)}
\newcommand\kozero{k_0(V)}
\newcommand\supi{\sup_{i \in \nn}}
\newcommand\sumi{\sum_{i=1}^\infty}
\newcommand\sumj{\sum_{j=1}^\infty}
\newcommand\limii{\lim_{i \to \infty}}
\newcommand\sumji{\sum_{j=1}^i}
\newcommand\sumjim{\sum_{j=1}^{i-1}}

\newcommand\dragilev{L_f(\alpha_i,\infty)}
\newcommand\igoes{\xrightarrow{i}}
\newcommand\lamball{B(\lambda,\delta)}
\newcommand\clamball{\overline{\lamball}}

\begin{document}

\author{Ersin Kızgut}

\address{Instituto Universitario de Matemática Pura y Aplicada \\ Universitat Politècnica de València \\ E-46071 Valencia, Spain}

\email{erkiz@upv.es}

\subjclass[2010]{47A10, 47B37, 46A45, 46A04}

\keywords{Cesàro operator, duals of smooth sequence spaces, generalized power series spaces, spectrum, (LB)-space.}

\title[The Cesàro operator]{The Cesàro operator on duals of smooth sequence spaces of infinite type}

\begin{abstract}
	The discrete Cesàro operator $\mathsf{C}$ is investigated in strong duals of smooth sequence spaces of infinite type. Of main interest is its spectrum, which turns out to be distinctly different in the cases when the space is nuclear and when it is not.
\end{abstract}

\maketitle

\section{Introduction}
	The discrete Cesàro operator $\ces$ acting on $\cnn$ is defined by
\[
	\ces x:=\left(x_1, \frac{x_1+x_2}{2}, \frac{x_1+x_2+x_3}{3}, \dots, \frac{x_1+\dots +x_i}{i}, \dots \right), \, x=(x_i)_{i \in \nn},
\]
which is a bicontinuous isomorphism of $\cnn$ onto itself with

\begin{equation}\label{ices}
	\ices y = (i y_i-(i-1)y_{i-1}), \quad y=(y_i) \in \cnn.
\end{equation}
For a diverse family of classical Banach spaces, the fundamental questions of continuity and determination of the spectrum have been investigated, and precise answers have been obtained. We refer the reader to the introduction of \cite{ABRwc0}. The behaviour of $\ces$ when acting on the Fréchet spaces $\cnn$, $\ell_{p+}=\bigcap_{q>p}\ell_q,\, 1 \leq p<\infty$, and on the strong duals $(\Lambda_\infty(\alpha))'_b$ of power series space of infinite type were studied in \cite{ABR13, ABR17, ABR18-2}. In this paper we generalize the results of \cite{ABR18-2} to the setting of the duals of certain types of Köthe echelon spaces, so called smooth sequence spaces of infinite type. These spaces were introduced by Terzioğlu \cite{Ter69, Ter73, Ter75, Ram79}. We also refer to Kocatepe \cite{Koc85, Koc88, Koc89}. The aim of this paper is to investigate the behaviour of $\ces$ when it acts on the strong duals $(\kotinf)'_b$ of smooth sequence spaces of infinite type. The reason for focusing on the infinite type dual spaces is that the Cesàro operator $\ces$ fails to be continuous on most of the finite type dual spaces (see Proposition~\ref{ces does not act in g1}). Some of our proofs are inspired by \cite{ABR18-2}, but new arguments are needed in this setting. We  distinctly expain our context. Let $A=(a_n)_n$, where $a_n=(a_n(i))_i$. $A$ is called a \textit{Köthe matrix} if the following conditions are satisfied:

\begin{enumerate}[label=(K\arabic*)]
	\item $0 \leq a_n(i) \leq a_{n+1}(i)$, for all $i, n \in \nn$.
	\item For all $i \in \nn$, there exists $n \in \nn$ such that $a_n(i)>0$.
\end{enumerate} 
The Köthe echelon space $\kotinf$ of order 1 is defined by
\begin{equation}
	\kotinf:= \{x \in \cnn: \sumi a_n(i)|x_i|<\infty, \, \forall n \in \nn \},
\end{equation}
which is a Fréchet space when equipped with the increasing system of seminorms
\[
	p_n(x) = \sumi a_n(i)|x_i|, \quad x \in \kotinf, \, n \in \nn.
\]
Then $\kotinf=\bigcap_{n \in \nn} \ell_1(a_n)$, where $\ell_1(a_n)$ is the usual Banach space. The space $\kotinf$ is given by the projective limit topology, that is, $\kotinf=\proj_n \ell_1(a_n)$. For the theory of Köthe echelon spaces $\lambda_p(A)$ of order $p$ for $1 \leq p \leq \infty$ or $p=0$, see \cite[Section 27]{Vog97}. Let $V=(v_n)_n=(\frac{1}{a_n})_n$. Then, the corresponding co-echelon space of $\kotinf$ is given by the (LB)-space $\koinf:=\ind_n \ell_\infty(v_n)$. For co-echelon spaces, the reader is referred to \cite{Bie88, BMS82, Kot69, Vog97}. A Köthe echelon space $\kotinf$ is said to be a \textit{smooth sequence space of infinite type} (or a $\ginf$-space) \cite[Section 3]{Ter69} if $A$ satisfies

\begin{enumerate}[label=(G$\infty$-\arabic*)]
	\item $1 \leq a_n(i) \leq a_n(i+1)$, for all $i,n \in \nn$.
	\item For all $n \in \nn$ there exist $m>n$ and $M>0$ such that $a_n(i)^2 \leq M a_m(i)$, for all $i \in \nn$.
\end{enumerate}

\begin{proposition}\textup{\cite[3.1]{Ter69}\cite[Theorem 4.9]{BMS82}}\label{kotinf always schwartz}
	For a $\ginf$-Köthe matrix $A$, the following statements are equivalent: 
	\begin{enumerate}[label=\normalfont(\arabic*)]
		\item $\kotinf$ is a Schwartz space.
		\item $\kotinf$ is not isomorphic to $\ell_1$.
		\item There exists $n \in \nn$ such that
		\[
		\limii a_n(i)=\infty. \tag{I}\label{tends to infinity}
		\]
		\item For all $n \in \nn$ there exists $m>n$ such that
		\[
			\limii \frac{v_m(i)}{v_n(i)} = \limii \frac{a_n(i)}{a_m(i)}=0. \tag{S}\label{condition S}
		\]
		\item $\koinf$ is isomorphic to $\kozero=\ind_n c_0(v_n)$.
		\item $\kozero$ is a Montel space.
	\end{enumerate}
	
\end{proposition}
In the light of Proposition~\ref{kotinf always schwartz}, we deal with the Cesàro operator $\ces$ defined on the co-echelon space $\kozero$ of order 0. Indeed, since the Köthe echelon space $\kotinf$ of order 1 is a Fréchet-Schwartz space (hence, distinguished) in our case, it follows that $\kozero=\ind_n c_0(v_n)=\ind_n \ell_\infty(v_n)=(\kotinf)'_b$ is the strong dual of $\kotinf$. For each $n \in \nn$ we define the norm
\begin{equation}
	q_n(x):=\supi v_n(i)|x_i|, \quad x=(x_i) \in \ell_\infty(v_n)
\end{equation}
whose restriction to $c_0(v_n)$ is the norm in $c_0(v_n)$. For each $n \in \nn$, $c_0(v_n) \subseteq c_0(v_m)$, for every $m \geq n$ and 
\begin{equation}\label{continuity by norm}
	q_m(x) \leq q_n(x), \quad \forall x \in c_0(v_n).
\end{equation}
Let us remind that nuclear spaces are in particular Schwartz. Since $\kozero=(\kotinf)'_b$, the nuclearity of $\kozero$ is equivalent to that of $\kotinf$. The following result is known (see \cite[3.1-b]{Ter69}), however, we give a partial proof.

\begin{proposition}\label{main equivalences 1}
	For a $\ginf$-Köthe matrix $A$, the following statements are equivalent:
	\begin{enumerate}[label=\normalfont(\arabic*)]
		\item For any $n \in \nn$ there exists $m>n$ such that
		\[
		\sumi\frac{v_m(i)}{v_n(i)}=\sumi \frac{a_n(i)}{a_m(i)} < \infty. \tag{GPC}\label{gpc}
		\]
		\item $\kotinf$ is nuclear.
		\item $\kozero$ is nuclear.
		\item There exists $n \in \nn$ such that
		 \[
		 \sumi v_n(i) < \infty. \tag{SV} \label{vn summable}
		 \]
	\end{enumerate}
\end{proposition}

\begin{proof}
	(1) $\Leftrightarrow$ (2) See e.g. \cite[Proposition 28.15]{Vog97}. 

	(2) $\Leftrightarrow$ (3) Follows by \cite[pp.78]{Pie72}.
		
	(3) $\Rightarrow$ (4) Since $\kotinf$ is a Schwartz space, by Proposition~\ref{kotinf always schwartz}, we may pick $n \in \nn$ as in \eqref{tends to infinity}. So there exists $M>0$ with $v_n(i) \leq M$, for all $i \in \nn$. Since $\kozero$ is nuclear, we may choose an $m>n$ as in \eqref{gpc}. Hence
	\[
	\sumi v_m(i) = \sumi v_n(i)\frac{v_m(i)}{v_n(i)} \leq M \sumi \frac{v_m(i)}{v_n(i)} < \infty.
	\]

	(4) $\Rightarrow$ (3) Suppose that there exists $n_0$ as in \eqref{vn summable}. For an $n \geq n_0$, if we pick $m>n$ and $C>0$ as in (G$\infty$-2), then
	\[
	\sumi \frac{v_m(i)}{v_n(i)} \leq C\sumi v_n(i) < \infty,
	\]
	since (G$\infty$-1) implies that $v_n(i)<v_{n_0}(i)$, for all $i \in \nn$. That means \eqref{gpc} is satisfied, and so $\kozero$ is nuclear. 
	\end{proof}
 A power series space $\pssi=\{x \in \cnn | \sumi \exp(\alpha_in)|x_i|<\infty, \, \forall n \in \nn\}$ of infinite type associated with the strictly  increasing sequence $\alpha_i \igoes \infty$ is a $\ginf$-space (see \cite[Section 29]{Vog97} for power series spaces of infinite type). But the converse is false, in general as shown in Example~\ref{ginf_non_pss} below. A Fréchet space $E$ equipped with the increasing system $(p_n(\cdot))_{n \in \nn}$ of seminorms is said to have property (DN) \cite[pp. 359]{Vog97} if there exists $s \in \nn$ such that for all $n \in \nn$ there exist $m \in \nn$ and $C>0$ satisfying
\[
	p_n(x)^2 \leq C p_m(x) p_s(x), \quad \forall x \in E. \tag{DN}\label{dn-condition}
\]
Here, such $p_s(\cdot)$ is a norm and is called the \textit{dominating norm}. It is straightforward to prove that a power series space $\pssi$ of infinite type satisfies property \eqref{dn-condition}. Example~\ref{ginf_non_pss} also illustrates that a $\ginf$-space satisfying condition \eqref{dn-condition} is still not necessarily isomorphic to a power series space of infinite type.

A \textit{Dragilev space of infinite type} $\dragilev$ is defined via $f\colon \rn \to \rn^+$ an odd, increasing, logarithmically convex (i.e., $\log \circ f$ is convex function for $x>0$), and the strictly increasing sequence $\alpha_i \igoes \infty$. If $a_n(i)=\exp(f(n\alpha_i))$, then $\dragilev$ is isomorphic to the Köthe echelon space $\kotinf$ of order 1. Let $0<\rho<\infty$, then the limit $\tau(\rho)=\lim_{x \to \infty} \frac{f(\rho x)}{f(x)} \leq \infty$ exists. The function $f$ is called rapidly increasing if $\tau(\rho)=\infty$, for all $\rho>1$. Otherwise, $f$ is called slowly increasing. In \cite[Section 3.2]{Dra83}, it is explained that the space $\dragilev$ is isomorphic to a power series space of infinite type if and only if $f$ is slowly increasing.
\begin{example}\label{ginf_non_pss}
	\normalfont Let $A$ be an infinite matrix defined by $a_n(i):=\exp(ine^{in})$. Then, the space $X:=\{x \in \cnn: \sumi a_n(i)|x_i|<\infty, \, \forall n \in \nn\}$, is a nuclear $\ginf$-space satisfying property \eqref{dn-condition} which is not isomorphic to a power series space of infinite type.
\end{example}
\begin{proof}
\begin{enumerate}[wide, label=(\roman*), labelwidth=!, labelindent=0pt]
\setlength\itemsep{0.5em}
	\item \textit{$X$ is a nuclear $\ginf$-space}: It is trivial to check that for any $n \in \nn$ one has $0<a_n(i)\leq a_m(i)$, for all $i \in \nn$ and for all $m\geq n$, so $X$ is a Köthe echelon space. It is also clear that for all $i \in \nn$ we have $1 \leq a_n(i)\leq a_n(i+1)$ so (G$\infty$-1) is satisfied. Now given $n \in \nn$, choose $m=2n$. Then $a_n(i)^2=\exp(2ine^{in})\leq \exp(2ine^{2in})=a_m(i)$ holds for all $i,n \in \nn$. So (G$\infty$-2) is also satisfied. Hence $X$ is a $\ginf$-space. For nuclearity, consider
	\[
	\log(i) < in \quad \Rightarrow \quad i<e^{in} <ne^{in} \quad \Rightarrow \quad -ne^{in} <-i.
	\]
	Given $n \in \nn$, select $m>n$ and $M>0$ as in (G$\infty$-2). Then,
	\[
	\sumi \frac{a_n(i)}{a_m(i)} \leq \sumi M\frac{a_n(i)}{a_n(i)^2} = M\sumi \frac{1}{a_n(i)} \leq M\sumi e^{-i^2} < \infty.
	\]
	Therefore, \eqref{gpc} is satisfied. So $X$ is nuclear, in particular Schwartz. 
	
	\item \textit{$X$ enjoys property \eqref{dn-condition}}: Without loss of any generality assume that $a_1(i):=1$. Let us pick $s=1$. For any $i,n \in \nn$, and for a constant $C>0$ we clearly have
	\begin{align*}
	& & \log(2in)+in & \leq  \log(2in)+2in \\
	&\Rightarrow & 2ine^{in} &\leq  2ine^{2in} + \log(C)\\
	&\Leftrightarrow & e^{2ine^{in}} &\leq  C e^{2ine^{2in}}\\
	&\Leftrightarrow & a_n(i)^2 & \leq C a_{2n}(i)a_1(i),
	\end{align*}
	With the choice $m=2n$, we conclude that $X$ has property \eqref{dn-condition}. 
	
	\item \textit{$X$ is not isomorphic to a power series space of infinite type}: Let us define $f\colon \rn \to \rn^+$ by
	\[
	 f(x):=\begin{cases}
		xe^{x}, & \text{ if } x\geq 0\\
		xe^{-x}, & \text{ if } x<0
	\end{cases}
	\]
Clearly $f$ is an odd, positive, increasing, and logarithmically convex function. Then, for $(\alpha_i)_{i \in \nn}= (i)_{i \in \nn}$, $X$ is isomorphic to the Dragilev space $\dragilev$ of infinite type. For any $1<\rho<\infty$, 
\[
\lim_{x \to \infty} \frac{f(\rho x)}{f(x)} = \lim_{x \to \infty} \frac{\rho x e^{\rho x}}{xe^x} = \lim_{x \to \infty} \rho e^{(\rho-1)x} = \infty.
\]
Hence, $f$ is rapidly increasing. By the comments prior to Example~\ref{ginf_non_pss}, $X$ cannot be isomorphic to a power series space of infinite type.
\end{enumerate}
\end{proof}

\section{Continuity and compactness of $\ces$ on $\kozero$}
An operator $T$ on a Fréchet space $X$ into itself is called bounded (resp. compact) if there exists a neighborhood $U$ of the origin of $X$ such that $TU$ is a bounded (resp. relatively compact) set in $X$. Recall that a Hausdorff inductive limit $E=\ind_n E_n$ of Banach spaces is called \textit{regular} if every bounded subset $B$ of $E$ is contained and bounded in some step $E_n$. The following lemma is well-known.

\begin{lemma}\label{inductive limit continuity}
	Let $E=\ind_m E_m$ and $F=\ind_n F_n$ be (LB)-spaces such that $E$ (resp. $F$) is the union of the sequence of Banach spaces $E_m$ (resp. $F_n$). Let $T\colon E \to F$ be a linear operator. Then 
	\begin{enumerate}[label=\normalfont(\arabic*)]
		\item $T$ is continuous if and only if for all $m \in \nn$ there exists $n \in \nn$ such that $T(E_m)\subset F_n$ and $T\colon E_m \to F_n$ is continuous.
		\item Let $T$ be continuous and let $F$ be regular. Then $T$ is bounded if and only if there exists $n \in \nn$ such that for all $m$, $T(E_m) \subset F_n$ and $T\colon E_m \to F_n$ is continuous.
	\end{enumerate}
\end{lemma}

\begin{proposition}\label{continuity}
	Let $\kotinf$ be a Schwartz Köthe echelon space of order 1. Then, $\ces\colon \kozero \to \kozero$ is continuous if and only if for all $n \in \nn$ there exists $m>n$ such that
	\begin{equation}\label{continuiuty criterion}
		\supi \frac{v_m(i)}{i}\sumji \frac{1}{v_n(j)} < \infty.
	\end{equation}
\end{proposition}

\begin{proof}
	Follows directly from Lemma~\ref{inductive limit continuity}, and \cite[Proposition 2.2(i)]{ABRwc0}.
\end{proof}

\begin{corollary}
	Let $A$ be a Köthe matrix satisfying (G$\infty$-1), and let $\kotinf$ be Schwartz. Then $\ces\in \clo(\kozero)$.
\end{corollary}

\begin{proof}
Since $\kotinf$ is Schwartz, for any $n \in \nn$ pick $m>n$ as in condition \eqref{condition S}. Hence, (G$\infty$-1) yields
\[
\supi \frac{v_m(i)}{i}\sumji \frac{1}{v_n(j)} \leq \supi\frac{iv_m(i)}{iv_n(i)} = \supi \frac{v_m(i)}{v_n(i)} < \infty.
\]
Thus, \eqref{continuiuty criterion} holds, and $\ces \in \clo(\kozero)$ by Proposition~\ref{continuity}.
\end{proof}

Now let us give a characterization for the compactness of $\ces$ in $\clo(\kozero)$. The following proposition is a direct consequence of Lemma~\ref{inductive limit continuity} and \cite[Proposition 2.2(ii)]{ABRwc0}, so we omit its proof.

\begin{proposition}\label{compactness}
	Let $A$ be a Köthe matrix satisfying (G$\infty$-1), and let $\kotinf$ be Schwartz. Then, $\ces\colon \kozero \to \kozero$ is compact if and only if there exists $n \in \nn$ such that for all $m>n$ one has
	\begin{equation}\label{compactness criterion}
		\limii \frac{v_m(i)}{i}\sumji \frac{1}{v_n(j)} =0.
	\end{equation}
\end{proposition}
The Köthe echelon space $\kotinf$ of order 1 is said to be a \textit{smooth sequence space of finite type} (or a $G_1$-space) \cite[Section 3]{Ter69} if $A$ satisfies
\begin{enumerate}[label=(G1-\arabic*)]
	\item $0 < a_n(i+1) \leq a_n(i)$, for all $n \in \nn$ and $i \in \nn$.
	\item For all $n \in \nn$ there exist $m>n$ and $C>0$ such that $a_n(i) \leq C {a_m(i)}^2$, for all $i \in \nn$.
\end{enumerate}
The Cesàro operator on $G_1$-spaces was studied by the author in \cite{Kiz18}. The following proposition shows that $\ces$ is not continuous on duals of nuclear $G_1$-spaces.
\begin{proposition}\label{ces does not act in g1}
	Let $A$ be a Köthe matrix satisfying (G1-1), and suppose $\kotinf$ is nuclear. Then, the Cesàro operator $\ces$ does not belong to $\clo(\kozero)$.
\end{proposition}

\begin{proof}
	Since $\kozero$ is nuclear, by \cite[Theorem 1]{Kiz18}, $\frac{i}{v_n(i)} \igoes 0$, for all $n \in \nn$. Now suppose $\ces$ is continuous on $\kozero$. Then by Proposition~\ref{continuity} and (G1-1), for $n=1$, there exists $m>1$ such that
\[
	M v_1(1) \geq v_1(1) \frac{v_m(i)}{i} \sumji \frac{1}{v_1(j)} \geq \frac{v_m(i)}{i v_1(1)}v_1(1) \igoes \infty,
\]
for a constant $M>0$. This is a contradiction. Hence, $\ces \notin \clo(\kozero)$.
\end{proof}
Let $D\colon\cnn \to \cnn$ be the formal operator of differentiation defined by $D(x_1, x_2, x_3, \dots) = (x_2, 2x_3, 3x_4, \dots)$, $x=(x_i)_i$. $D$ is closely related to the Cesàro operator $\ces \in \clo(\cnn)$ by the identity $\ices=(I-S_r)\circ D\circ S_r$, where $S_r \in \clo(\cnn)$ is the right-shift operator. The following result is proved via a similar argument in \cite[Proposition 7]{Kiz18}.
\begin{proposition}
	Let $A$ be a Köthe matrix. Then, for the co-echelon space $\kozero$ and the formal differentiation operator $D$, the following statements are equivalent:
	\begin{enumerate}[label=\normalfont(\arabic*)]
		\item The differentiation operator $D\colon \kozero \to \kozero$ is continuous.
		\item For all $n \in \nn$ there exist $m>n$ and $M>0$ such that
		\[
			\supi iv_m(i)\leq Mv_n(i+1).
		\]
	\end{enumerate}
\end{proposition}

\begin{example}
	\normalfont Consider the nuclear $\ginf$-space $X$ constructed in Example~\ref{ginf_non_pss}. For $Y:=k_0(V)=(\kotinf)'=X'$, choose $m=2n$ to observe
	\begin{align*}
	\frac{iv_m(i)}{v_n(i+1)} = \frac{i\exp(-2ine^{2in})}{\exp(-ine^{(i+1)n})}  & = \exp(\log(i)+ine^{(i+1)n}-2ine^{2in}) \\
					& = \exp\left(\left(\frac{\log(i)}{ine^{2in}}+e^ne^{-in}-2\right)ine^{2in}\right) \igoes 0.
	\end{align*}
	Hence, $D\colon Y \to Y$ is continuous.
\end{example}

\begin{proposition}\label{characterization nuclearity} 
	For a $\ginf$-Köthe matrix $A$, and the associated co-echelon space $\kozero$, the following statements are equivalent:
	\begin{enumerate}[label=\normalfont(\arabic*)]
		\item $\kozero$ is nuclear.
		\item For all $n \in \nn$ there exists $m>n$ such that
		\[
			\supi \frac{iv_m(i)}{v_n(i)} < \infty. \tag{N}\label{nuclearity criterion}
		\]
		\item Given $\alpha \in \rn$, for all $n \in \nn$ there exists $m>n$ such that
		\[
			\supi \frac{i^\alpha v_m(i)}{v_n(i)} < \infty. \tag{SN}\label{strong nuclearity criterion}
		\]
	\end{enumerate}
\end{proposition}

\begin{proof}
The proof reads as \cite[Proposition 9]{Kiz18}. In the implication (1) $\Rightarrow$ (2) if we set $b_n(i):=\prod_{j=1}^n a_j(i)$, the rest follows with the same arguments.
\end{proof}

\begin{proposition}\label{nuclear rapidly decreasing}
	Given a real number $\alpha \geq 1$. Then, for a $\ginf$-Köthe matrix $A$, and the associated co-echelon space $\kozero$, the following statements are equivalent:
	\begin{enumerate}[label=\normalfont(\arabic*)]
		\item $\kozero$ is nuclear.
		\item There exists $n \in \nn$ such that $\limii i^\alpha v_n(i)=0$.
		\item There exists $n \in \nn$ such that $\limii i^\alpha v_n(i)=L<\infty$.
	\end{enumerate}
\end{proposition}

\begin{proof}
	(1) $\Rightarrow (2)$ Let $\kozero$ be nuclear. For any $n \in \nn$ pick $m>n$ as in \eqref{nuclearity criterion}. Then, by (G$\infty$-1) and proof of \cite[Proposition 9]{Kiz18}, we have
	\[
	i^\alpha v_m(i) \leq \frac{i^\alpha v_m(i)}{v_n(i)} \igoes 0.
	\]
	
	(2) $\Rightarrow$ (3) Trivial.
	
	(3) $\Rightarrow$ (1) Let us have $n \in \nn$ with $i^\alpha v_n(i) \igoes L$, for a given $\alpha>1$. For this $n$ we pick $m>n$ and $C>0$ as in (G$\infty$-2) and similarly we find $p>m$ and $M>0$ as in (G$\infty$-2). Then, for all $i \in \nn$,
	\[
	\frac{v_p(i)}{v_n(i)} \leq C\frac{v_p(i)}{v_m(i)} \leq CM\frac{v_m(i)^2}{v_m(i)} \leq CMv_n(i) \leq \frac{CML}{i^\alpha} \in \ell_1. 
	\]
	Hence, $(\frac{v_p(i)}{v_n(i)})_i \in \ell_1$ and so \eqref{gpc} is satisfied. Therefore, $\kotinf$ is nuclear.
\end{proof}

\section{Spectrum of $\ces$ in the nuclear case}
For a locally convex Hausdorff space $X$ and $T \in \clo(X)$, the \textit{resolvent set} $\rho(X;T)$ of $T$ consists of all $\lambda \in \cn$ such that $(\lambda I-T)^{-1}$ exists in $\clo(X)$. The set $\sigma(T;X):=\cn\setminus\rho(T;X)$ is called the \textit{spectrum} of $T$ on $X$. The \textit{point spectrum} $\sigma_{pt}(T;X)$ of $T$ on $X$ consists of all $\lambda \in \cn$ such that $(\lambda I-T)$ is not injective. Unlike for Banach spaces, it might happen that $\rho(T;X)=\varnothing$ or that $\rho(T;X)$ is not open in $\cn$. That is why some authors prefer the subset $\rho^*(T;X)$ of $\rho(T;X)$ consisting of all $\lambda \in \cn$ for which there exists $\delta>0$ such that the open disk $\lamball:=\{z \in \cn:|z-\lambda|<\delta\}\subseteq \rho(T)$ and $\{R(\mu,T):\mu \in \lamball\}$ is equicontinuous in $\clo(X)$. We denote $\Sigma:=\{\frac{1}{m}\colon m \in \nn\}$ and $\Sigma_0:=\Sigma \cup \{0\}$. In this section we investigate the spectra $\ptsp$, $\stsp$, and $\spec$ in case $\kozero$ is nuclear (equivalently, $\kotinf$ is nuclear).

\begin{proposition}\label{nuclear equivalent invertible}
	Let $A$ be a Köthe matrix satisfying (G$\infty$-1). Then, the following statements are equivalent:
	\begin{enumerate}[label=\normalfont(\arabic*)]
		\item $0 \notin \spec$.
		\item $A$ satisfies condition \eqref{nuclearity criterion}.
\end{enumerate}
\end{proposition}

\begin{proof}
	$0 \notin \spec$ if and only if $\ices \colon \kozero \to \kozero$ is continuous if and only if for all $n \in \nn$ there exists $m>n$ such that $\ices \colon c_0(v_n) \to c_0(v_m)$ is continuous. Hence the proof proceeds as in \cite[Proposition 10]{Kiz18}.
\end{proof}

\begin{lemma}\textup{\cite[Proposition 2.6]{ABRwc0}}\label{1/s in ptsp}
	Let $A$ be a Köthe matrix satisfying (G$\infty$-1), and let $\kotinf$ be Schwartz. Then for $s \in \nn$ and for the Cesàro operator $\ces$, the following statements are equivalent:
	\begin{enumerate}[label=\normalfont(\arabic*)]
		\item $\frac{1}{s+1} \in \ptsp$.
		\item There exists $n \in \nn$ such that
		$ \limii i^{s}v_n(i)=0$.
	\end{enumerate}
\end{lemma}

\begin{proposition}
Let $A$ be a (G$\infty$-1) Köthe matrix satisfying \eqref{nuclearity criterion}, and let $\kotinf$ be Schwartz. Then, $\Sigma=\ptsp$.	
\end{proposition}

\begin{proof}
	We clearly have $\ptsp \subseteq \sigma_{pt}(\ces,\cnn)=\Sigma$. Now we prove that there exists $n \in \nn$ such that for all $s\in \nn$ we have $i^s v_n(i)\igoes 0$, by induction over $s$. Since $\kotinf$ is Schwartz, we may choose $n \in \nn$ as in condition \eqref{tends to infinity} so that $v_n(i) \igoes 0$. Then, by assumption there exist $m_1>n$ and $M>0$ such that $iv_{m_1}(i) \leq Mv_n(i)$ and hence $iv_{m_1}(i) \igoes 0$. Suppose that $i^s v_{m_1}(i)\igoes 0$, for $s=1,\dots, r$. Then, there exist $m_2>m_1$ and $\tilde{M}$ satisfying
	\[
	i^{r+1}v_{m_2}(i)=i^r iv_{m_2}(i) \leq \tilde{M}i^r v_{m_1}(i) \igoes 0.
	\]
	That implies for some $n \in \nn$, we have $i^s v_n(i)\igoes 0$, for all $s\in \nn$. Equivalently, by Lemma~\ref{1/s in ptsp},  $\frac{1}{s+1} \in \ptsp$, for all $s \in \nn$. Therefore $\Sigma=\ptsp$.
\end{proof}

	\begin{theorem}\label{main equivalences 2}
		Let $\kotinf$ be a $\ginf$-space which is Schwartz. Then, the following statements are equivalent:
		\begin{enumerate}[label=\normalfont(\arabic*)]
			\item $0 \notin \spec$.
			\item $\frac{1}{2} \in \ptsp$.
			\item There exists $s \in \nn$ such that $\frac{1}{s} \in \ptsp$.
			\item $\Sigma=\ptsp$.
		\end{enumerate}
	\end{theorem}
	
	\begin{proof}
		(1) $\Rightarrow$ (2) Proposition~\ref{characterization nuclearity} and Proposition~\ref{nuclear equivalent invertible} yield $\kozero$ is nuclear. Then, by Proposition~\ref{main equivalences 1} we may take $n \in \nn$ as in \eqref{vn summable}, so that $(v_n(i))_i \in \ell_1$. We may also pick $m>n$ and $M>0$ as in \eqref{nuclearity criterion} so that $iv_m(i) \leq Mv_n(i)$, for all $i \in \nn$. Then, we have $(iv_m(i))_i \in \ell_1$. This implies $iv_m(i) \igoes 0$. This is equivalent to (2) by Lemma~\ref{1/s in ptsp}.
		
		(2) $\Rightarrow$ (1) Since $\frac{1}{2} \in \ptsp$, by Lemma~\ref{1/s in ptsp} there is an $n \in \nn$ such that $iv_n(i) \igoes 0$. If we select $m>n$ and $M>0$ as in (G$\infty$-2), we obtain
		\[
		\frac{iv_m(i)}{v_n(i)}\leq Miv_n(i) \igoes 0.
		\]
		This is equivalent to (1) by Proposition~\ref{nuclear equivalent invertible}.
		
		(2) $\Rightarrow$ (3) Clear.
		
		(3) $\Rightarrow$ (4) By Lemma~\ref{1/s in ptsp} we have an $s\in \nn$ with $i^sv_n(i) \igoes 0$ for some $n \in \nn$. Then, clearly $iv_n(i) \igoes 0$ as well. Now let us prove by induction that $i^{2^k}v_n(i) \igoes 0$ for all $k \in \nn$. For $k=0$, it is already satisfied. Suppose $i^{2^k}v_n(i) \igoes 0$, for $k=1,\dots,r$. Then, for some $i_0 \in \nn$, we have $|i^{2^r}v_n(i)|<1$ for all $i\geq i_0$. For $m>n$ and $C>0$ selected as in (G$\infty$-2) and for all $i \geq i_0$,
		\[
			i^{2^{r+1}}v_m(i)=(i^{2^r}\sqrt{v_m(i)})^2 \leq C(i^{2^r} v_n(i))^2 \leq Ci^{2^r}v_n(i) \igoes 0.
		\]
		Therefore $\frac{1}{s} \in \ptsp$ for all $s \in \nn$. Hence, $\ptsp=\Sigma$.
		
		(4) $\Rightarrow$ (2) Trivial.
	\end{proof}
The following example illustrates why assumption (G$\infty$-2) in Theorem~\ref{main equivalences 2} cannot be removed.
\begin{example}\label{second axiom indispensible}
\normalfont
\begin{enumerate}[wide, label=\normalfont(\roman*), labelwidth=!, labelindent=0pt]
\setlength\itemsep{0.5em}
	\item For a fixed $0<\alpha<1$, and an increasing sequence $(\alpha_n)_n \subset (0,1)$ tending to $\alpha$, let us define a Köthe matrix $A$ by $a_n(i):=i^{\alpha_n}e^i$, where $i,n \in \nn$. The Köthe echelon space $\kotinf$ of order 1 satisfies condition (G$\infty$-1) and condition \eqref{tends to infinity}. Assume that (G$\infty$-2) also holds. Then, given $n=1$ there is $m>1$ and $M>0$ with
\[
i^{2\alpha_1}e^{2i} \leq Mi^{\alpha_n}e^i \quad \Rightarrow \quad i^{\alpha_1}e^i \leq M i^{\alpha_n-\alpha_1},
\]
which is impossible. Hence $A$ is not a $\ginf$-matrix. For $n=1$,
\[
\supi \frac{i^{1+\alpha_1}e^i}{i^{\alpha_m}e^i}=\supi i^{1+\alpha_1-\alpha_m}=\infty, \quad \forall m>1,
\] 
since $1+\alpha_1-\alpha_m>0$, for all $m>1$. So \eqref{nuclearity criterion} is not satisfied. On the other hand, for each $s,n \in \nn$ and for large values of $i \in \nn$, we have $0<i^{s-1}v_n(i)=i^{s-1-\alpha_n}e^{-i} \leq i^{s-1}e^{-i} \igoes 0$. So $(i^{s-1})_i \in k_0(V)$ and by Lemma~\ref{1/s in ptsp} $\frac{1}{s} \in \ptsp$ for each $s \in \nn$. This shows that condition (4) does not imply condition (1) in Theorem~\ref{main equivalences 2}, in general.
	
	\item Fix $s \geq 1$, $s \in \nn$ and define the Köthe matrix $A=(a_n)_n$ by $a_n(i):=i^{s-\frac{1}{1+n}}$. The Köthe echelon space $\kotinf$ of order 1 satisfies (G$\infty$-1) and condition \eqref{tends to infinity}, but it is not a $\ginf$-space. Indeed, assume that (G$\infty$-2) holds. Then for $n=1$, there exist $m>1$ and $M>0$ such that $a_1(i)^2 \leq Ma_m(i)$. So for any $s \geq 1$,
\[
i^{2s-1} \leq M i^{s-\frac{1}{1+m}} \quad \Rightarrow \quad i^s \leq M i^{1-\frac{1}{1+m}}.
\]
But this is impossible since $s \geq 1$. In this case, $(i^{m-1})_i \in k_0(V)$ for $m=1,2, \dots s$ but $(i^s)_i \notin k_0(V)$ since $i^sv_n(i)=i^{\frac{1}{1+n}} \igoes \infty$, for all $n \in \nn$. Thus $\frac{1}{s+1} \notin \ptsp$, which implies $\frac{1}{m}\notin \ptsp$ for each $m>s$. This shows us condition (3) in Theorem~\ref{main equivalences 2} does not imply (4), in general.
\end{enumerate}
\end{example}
Let us define the continuous function $a:\cn\setminus\{0\}\to \rn$ by 
\begin{equation}\label{function a}
a(z):=\operatorname{Re}\left(\frac{1}{z}\right).
\end{equation}
Observe that for all $k \in \nn$, the weighted Banach space $c_0(v_k)$ is isometrically isomorphic to $c_0$ via $\phi_k\colon c_0(v_k)\to c_0$ defined by 
\begin{equation}\label{c0 isom weighted}
\phi_k x:=(v_k(i)x_i)_i, \quad \forall x \in c_0(v_k).
\end{equation}
\begin{proposition}\label{spectrum nuclear case}
	Let $\kotinf$ be a nuclear $\ginf$-space. Then,
	\begin{enumerate}[label=\normalfont(\arabic*)]
		\item $\spec=\ptsp=\Sigma$.
		\item $\stsp=\spec \cup \{0\}=\Sigma_0$.
	\end{enumerate} 	
\end{proposition}
\begin{proof}
	Since $\kozero$ is nuclear, by Theorem~\ref{main equivalences 2}, we know that $\ptsp=\Sigma\subseteq\spec$, and hence
\[
\Sigma_0=\overline{\Sigma} \subseteq \overline{\spec} \subseteq \stsp.
\]
Moreover, Proposition~\ref{nuclear equivalent invertible} yields $0 \notin \spec$. For the other inclusion, we show that for every $\lambda \in \cn\setminus\Sigma_0$ there exists $\delta>0$ such that the inverse operator $(\ces-\mu I)^{-1}\colon \kozero \to \kozero$ is continuous for each $\mu \in \lamball$ and the set $\{(\ces-\mu I)^{-1}\colon\mu \in \lamball\}$ is equicontinuous in $\clo(\kozero)$. Remember that $(\ces-\mu I)^{-1}$ is continuous on $\cnn$ for each $\mu \in \cn\setminus\Sigma$. Fix $\lambda \in \cn\setminus\Sigma_0$. Choose a $\delta_1>0$ such that $\lamball\cap\Sigma_0=\varnothing$. To establish our claim, it suffices to show that there exists $\delta>0$ such that for all $n \in \nn$ there exist $m>n$ and $D_n>0$ satisfying
	\begin{equation}\label{continuity out of sigma}
		q_m((\ces-\mu I)^{-1}x) \leq D_nq_n(x), \quad \forall \mu \in \lamball,\, x \in c_0(v_n).
	\end{equation}
	Now we separate in two cases.
\begin{enumerate}[wide, label=(\roman*), labelwidth=!, labelindent=0pt]
\setlength\itemsep{0.5em}	
	\item $a(\lambda)<1$ (equivalently, $|\lambda-\frac{1}{2}|>\frac{1}{2})$: Fix $n\in \nn$. Since $a(\lambda)<1$ we may pick $\varepsilon>0$ such that $a(\lambda)<1-\varepsilon$. By continuity of $a$, there exists $\delta_2>0$ such that $a(\mu)<1-\varepsilon$, for all $\mu \in B(\lambda,\delta_2)$. By \cite[Lemma 2.8]{ABR18-2} , there exist $\delta\in (0,\delta_2)$ and $M_{n,\lambda}$ such that \eqref{continuity out of sigma} is satisfied:
	\[
		q_n((\ces-\mu I)^{-1}x) \leq \frac{M_{n,\lambda}}{1-a(\mu)}q_n(x) \leq\frac{M_{n,\lambda}}{\varepsilon}q_n(x),\quad \forall \mu \in \overline{\lamball},\, x \in c_0(v_n).
	\]
	\item $a(\lambda)\geq 1$ (equivalently, $|\lambda-\frac{1}{2}|\leq \frac{1}{2})$: Let us recall the formula for the operator $(\ces-\mu I)^{-1}\colon\cnn \to \cnn$ whenever $\mu \notin \Sigma_0$. By \cite{Rea85}, the $i$-th row of the matrix for $(\ces-\mu I)^{-1}$ has the entries:
	 \[
		   	\begin{cases}
		   	\displaystyle	\frac{-1}{i \mu^2 \prod_{k=j}^i (1-\frac{1}{k\mu})} & \text{if } 1 \leq j <i \vspace{0.25cm}\\
		   	\displaystyle	\frac{1}{\frac{1}{i}-\mu} & \text{if } i=j \vspace{0.25cm}\\
		   		0 & \text{otherwise}.
		   	\end{cases}
		   \]
 For $D_\mu=(d_{ij})_{i,j}$ and $E_\mu=(e_{ij})_{i,j}$, one may formulate $(\ces-\mu I)^{-1}=D_\mu-\frac{1}{\mu^2}E_\mu$, where $	d_{ij}=\frac{1}{\frac{1}{i}-\mu}$, for $i=j$ otherwise $d_{ij}=0$; and $e_{ij}=\frac{1}{i\prod_{k=j} (1-\frac{1}{k\mu})}$, for $2 \leq j <i$ otherwise $e_{ij}=0$. Define $d_0(\lambda):=\operatorname{dist}(\lamball,\Sigma_0)>0$. We have $|d_{ii}|<\frac{1}{d_0(\lambda)}$, for all $\mu \in \overline{B(\lambda,\delta_1)}$ and $i\in \nn$. Fix $n \in \nn$. Then for every $x \in c_0(v_n)$ and $\mu \in \overline{B(\lambda,\delta_1)}$
\[
q_n(D_\mu(x))=\sumi|d_{ii}(\mu)x_i|v_n(i) \leq \frac{1}{d_0(\lambda)}\sumi |x_i|v_n(i)=\frac{1}{d_0(\lambda)}q_n(x).
\]
So $\{D_\mu:\mu\in \overline{B(\lambda,\delta_1)}\}\subseteq\clo(c_0(v_m))$. Then, it remains to show that $E_\mu\colon \kozero \to \kozero$ is continuous for all $\mu \in \clamball$ for some $\delta>0$. So by \eqref{c0 isom weighted}, it suffices to show that for all $n \in \nn$ there exist $m\geq n$ and $D_n>0$ such that
\begin{equation}
	\|\phi_m \circ E_\mu \circ \phi_n^{-1}x\|_0 \leq D_n \|x\|_0, \quad \forall x \in c_0,\, \mu \in \overline{B(\lambda,\delta_1)}.
\end{equation}
where $\|\cdot\|_0$ is the usual $c_0$-norm. For each $n,m$ define $\tilde{E}_{\mu,n,m}:=\phi_m\circ E_\mu \circ \phi_n^{-1}\in \clo(\cnn)$ for $\mu \in \cnn\setminus\{0\}$. Fix $n\in \nn$. For each $m\geq n$ the operator $\tilde{E}_{\mu,m,n}$ for $\mu \in B(\lambda,\delta_1)$ is the restriction to $c_0$ of
\[
\tilde{E}_{\mu,m,n}(x)=(\tilde{E}_{\mu,m,n}(x))=\left(v_m(i)\sumjim\frac{e_{ij}(\mu)}{v_n(j)}x_j\right), \quad x \in \cnn,
\]
with $(\tilde{E}_{\mu,m,n})_1:=0$. $\tilde{E}_{\mu,m,n}=(\tilde{e}_{ij}^{nm}(\mu))$ is given by $\tilde{e}_{1j}^{nm}=0$, $\tilde{e}_{ij}^{nm}=\frac{v_m(i)}{v_n(j)}e_{ij}(\mu)$ for $i\geq 2$ and $1 \leq j<i$. So it suffices to verify, for some $m\geq n$ and $\delta>0$ one has $\tilde{E}_{\mu,m,n}\in \clo(c_0)$ for $\mu \in \lamball$, and $\{\tilde{E}_{\mu,m,n}\colon \mu \in \lamball\}$ is equicontinuous in $\clo(c_0)$. To prove this, we observe \cite[Lemma 2.7]{ABR18-2} implies that for every $m\geq n$, and all $\mu \in \overline{B(\lambda,\delta_2)}$ that
\begin{equation}\label{nuclear entry inequality}
	|\tilde{e}_{ij}^{nm}(\mu)|=\frac{v_m(i)}{v_n(j)}|e_{ij}(\mu)| \leq D'_\lambda\frac{i^{a(\mu)-1}v_m(i)}{j^{a(\mu)}v_n(j)},
\end{equation}
for some $D'_\lambda>0$ and $\delta_2 \in (0,\delta_1)$. Since $a$ is continuous, there exists $\delta\in (0,\delta_2)$ such that $a(\lambda)-\frac{1}{2}<a(\mu)<a(\lambda)+\frac{1}{2}$ for all $\mu \in \overline{\lamball}$. Then $a(\mu)>a(\lambda)-\frac{1}{2}\geq \frac{1}{2}$. By picking $m>n$ and $M>0$ as in \eqref{strong nuclearity criterion}, for any $\mu \in \clamball$ we have
\begin{equation}\label{nuclear condition 1}
|\tilde{e}_{ij}^{nm}(\mu)| \leq D'_\lambda \frac{v_m(i)}{v_n(j)}\frac{i^{a(\mu)-1}}{j^{a(\mu)}} \leq \frac{D'_\lambda}{j^{a(\mu)}} \frac{i^{a(\mu)}v_m(i)}{iv_n(i)} \leq \frac{MD'_\lambda}{i} \igoes 0,
\end{equation}
Moreover, employing \eqref{nuclear entry inequality} and (G$\infty$-1), respectively, we obtain
\begin{align}\label{nuclear condition 2}
	\supi\sumj|\tilde{e}_{ij}^{nm}(\mu)| &\leq \supi D'_\lambda i^{a(\mu)-1}\sumjim\frac{v_m(i)}{j^{a(\mu)}v_n(j)}\leq\supi D'_\lambda i^{a(\mu)-1}v_m(i)\frac{i-1}{v_n(i)} \nonumber\\
	& \leq\supi D'_\lambda \frac{i^{a(\mu)}v_m(i)}{v_n(i)} < \infty,
\end{align}
for every $\mu \in \clamball$. Hence, \cite[Lemma 2.1]{ABRwc0} implies that satisfying both \eqref{nuclear condition 1} and \eqref{nuclear condition 2} yields  $\tilde{E}_{\mu,m,n} \in \clo(c_0)$ for all $\mu \in \clamball$. Moreover, the operator norm is given by $\|\tilde{E}_{\mu,m,n}\| = \supi \sumji |\tilde{e}_{ij}^{nm}(\mu)|$, and we have shown that there exists $C>0$ such that $\|\tilde{E}_{\mu,m,n}\| \leq CD'_\lambda$, for all $\mu \in \clamball$. This implies $\{\tilde{E}_{\mu,m,n}\colon \mu \in \clamball\}$ is equicontinuous in $\clo(c_0)$.
\end{enumerate}
\end{proof}

\begin{corollary}
	Let $\kotinf$ be a nuclear $\ginf$-space. Then $\ces \in \clo(\kozero)$ is neither compact nor weakly compact.
\end{corollary}
\begin{proof}
	Since $\kozero$ is a Montel space, there is no distinction between compactness and weak compactness. So, suppose $\ces$ is compact. Then $\spec$ is necessarily a compact set in $\cn$ \cite[Theorem 9.10.2]{Edw65}. This contradicts Proposition~\ref{spectrum nuclear case}.
\end{proof}

When acting on $\cnn$, the Cesàro matrix $\ces$ is similar to the diagonal matrix $\operatorname{diag}(\frac{1}{i})$. Indeed, the identity $\ces=\Delta\operatorname{diag}(\frac{1}{i})\Delta$ holds in $\clo(\cnn)$, where
\[
\Delta=\Delta^{-1}=(\Delta_{ij})_{i,j \in \nn}=
\begin{cases}
\displaystyle
	(-1)^{j-1}\binom{i-1}{j-1}, &\text{ if } 1 \leq j <i \vspace{0.1cm} \\
	0, & \text{ if } j>i
\end{cases}
\]
and all the three operators $\ces$, $\operatorname{diag}(\frac{1}{i})$, and $\Delta$ are continuous.
\begin{proposition}\label{delta continuous}
	For a $\ginf$-Köthe matrix $A$, and for the operator $\Delta$, the following statements are equivalent:
	\begin{enumerate}[label=\normalfont(\arabic*)]
		\item There exists $n \in \nn$ such that 
		\[
		\supi i^i v_n(i)<\infty. \tag{U} \label{ultranuclear}
		\]
		\item $\Delta \in \clo(\kozero)$.
	\end{enumerate}
\end{proposition}

\begin{proof}
	For every $k \in \nn$, the surjective isomorphism $\phi_k\colon c_0(v_k) \to c_0$ is defined by \eqref{c0 isom weighted}. Since $\kozero=\ind_n c_0(v_n)$, we have $\Delta \in \clo(\kozero)$ if and only if for all $n \in \nn$ there exists $m>n$ with $\Delta\colon c_0(v_n) \to c_0(v_m)$ is continuous if and only if the operator $D^{nm}\colon c_0 \to c_0$ defined by $D^{nm}:=\phi_m \circ \Delta \circ \phi^{-1}_n$ is continuous, where $\phi_m=\operatorname{diag}(v_m(i))$ and $\phi^{-1}_n=\operatorname{diag}(\frac{1}{v_n(i)})$. Hence, $D^{nm}$ has a lower triangular matrix whose entries are given by
	\[
	d^{nm}_{ij}=(-1)^{j-1}\frac{v_m(i)}{v_n(j)}\binom{i-1}{j-1}, \quad 1\leq j<i
	\]
	and $d^{nm}_{ij}=0$ for $j\geq i$. It follows by \cite[Theorem 4.51-C]{Tay58} that $\Delta \in \clo(\kozero)$ if and only if for each $n \in \nn$ we find $m>n$ so that both \eqref{delta limit 0} and \eqref{delta summable} hold:
	\begin{equation}\label{delta limit 0}
		\limii |d^{nm}_{ij}|=\limii \frac{v_m(i)}{v_n(j)} \binom{i-1}{j-1}=0, \quad \forall j \in \nn,
	\end{equation}
	\begin{equation}\label{delta summable}
		\supi \sumj |d^{nm}_{ij}| = \supi \sumji \frac{v_m(i)}{v_n(j)} \binom{i-1}{j-1} <\infty
	\end{equation}
Observe that
\begin{align}\label{binomial estimation}
	\binom{i-1}{j-1} &=\frac{(i-1)!}{(j-1)!(i-j)!}=\frac{(i-1)\cdots(i-j+1)}{(j-1)!} \nonumber \\
	& = \frac{i^{j-1}}{(j-1)!}\left(1-\frac{1}{i}\right) \cdots \left(1-\frac{j-1}{i}\right).
\end{align}

(1) $\Rightarrow$ (2) Let us assume that there exists $n_0 \in \nn$ as in condition \eqref{ultranuclear}. Then, $\supi i^i v_n(i)<\infty$, for every $n \geq n_0$. In particular, $\limii i^\alpha v_n(i)=0$, for a given real number $\alpha>1$. First using (G$\infty$-1) then \eqref{binomial estimation} and then given $n \geq n_0$, taking $m>n$ and $C>0$ as in (G$\infty$-2) yield
	\[
	\frac{v_m(i)}{v_n(j)}\binom{i-1}{j-1} \leq \frac{v_m(i)}{v_n(i)}\binom{i-1}{j-1} \leq Cv_n(i) \frac{i^{j-1}}{(j-1)!} \left(1-\frac{1}{i}\right) \cdots \left(1-\frac{j-1}{i}\right)\igoes 0,
	\]
	for all $j \in \nn$. This shows that \eqref{delta limit 0} is satisfied. To prove that \eqref{delta summable} also holds, we first use (G$\infty$-1), then given $n \geq n_0$ we choose $m>n$ and $\tilde{C}>0$ as in (G$\infty$-2), and then apply \eqref{binomial estimation} to get
\begin{align*}
\supi\sumji \frac{v_m(i)}{v_n(j)}\binom{i-1}{j-1} & \leq \supi \frac{v_m(i)}{v_n(i)}\sumji \binom{i-1}{j-1} \\
&  \leq \tilde{C}\supi v_n(i) \sumji \left[\frac{i^{j-1}}{(j-1)!}\left(1-\frac{1}{i}\right) \cdots \left(1-\frac{j-1}{i}\right)\right]\\
& \leq \tilde{C}\supi i^i v_n(i) \left(1-\frac{1}{i}\right)  <\infty,
\end{align*}
by (1). Therefore, $\Delta \in \clo(\kozero)$.

 (2) $\Rightarrow$ (1) Suppose $\Delta \in \clo(\kozero)$. First we apply \eqref{delta summable}, and then (G$\infty$-1) along with \eqref{binomial estimation} to proceed
 \begin{align*}
 S&\geq \supi \sumji \frac{v_m(i)}{v_n(j)}\binom{i-1}{j-1} \geq \supi \frac{v_m(i)}{v_n(1)} \sumji \left[\frac{i^{j-1}}{(j-1)!}\left(1-\frac{1}{i}\right) \cdots \left(1-\frac{j-1}{i}\right)\right] \\
 &=\supi \frac{v_m(i)}{v_n(1)}i^i \sumji \left[\frac{1}{(j-1)!i^{i-j+1}}\left(1-\frac{1}{i}\right) \cdots \left(1-\frac{j-1}{i}\right)\right],
  \end{align*}
  for a constant $S>0$. Since for any $j$,
 \[
 \sumji \left[\frac{1}{(j-1)!i^{i-j+1}}\left(1-\frac{1}{i}\right) \cdots \left(1-\frac{j-1}{i}\right)\right] \igoes 0,
 \]
 and $S<\infty$, one has $\supi i^i v_m(i) <\infty$.
\end{proof}

\begin{remark}\label{stronger than nuclearity}
	\normalfont Obviously, condition \eqref{ultranuclear} implies nuclearity. However, the converse is not true, in general. Indeed, let $a_n(i):=\exp(in)$, for $i,n \in \nn$. Then, it is straightforward to show that $A=(a_n)_n$ is a $\ginf$-matrix. Moreover, since
\[
2\log(i) < in \quad \Rightarrow \quad i^2<e^{in} \quad \Rightarrow \quad e^{-in} <\frac{1}{i^2},
\]
for all $i,n \in \nn$, if we choose $m>n$ and $C>0$ as in (G$\infty$-2) we obtain
\[
\sumi \frac{v_m(i)}{v_n(i)} \leq C\sumi \frac{v_n(i)^2}{v_n(i)} = C\sumi v_n(i) \leq C\sumi \frac{1}{i^2} < \infty.
\]
Hence $\kozero$ is nuclear. On the other hand, we directly observe that for every $n \in \nn$, $\supi \frac{i^i}{e^{in}} =\infty$, which means the failure of condition \eqref{ultranuclear}.
\end{remark}

\section{The spectrum of $\ces$ in the non-nuclear case}

In this section we give a description of the spectra $\ptsp$ and $\spec$ when $\kozero$ is not nuclear (equivalently, $\kotinf$ is not nuclear). The following proposition is immediate from previous section.

\begin{proposition}\label{nonnuclear consequence}
	Let $\kotinf$ be a $\ginf$-space which is Schwartz. Then, the following statements are equivalent:
	\begin{enumerate}[label=\normalfont(\arabic*)]
		\item $\kozero$ is not nuclear.
		\item $\ptsp=\{1\}$.
		\item $0 \in \spec$.
	\end{enumerate}
\end{proposition}
Since $\ces \in \clo(\kozero)$, its dual $\ces'$ is defined and continuous on $\kozero'$ and is given by the formula
	\begin{equation}
		\ces'y:=\left(\sum_{j=i}^\infty \frac{y_j}{j}\right)_{i \in \nn}, \quad y=(y_i) \in \kozero';
	\end{equation}
see \cite[pp.774]{ABRwc0}. The following lemma is well-known. For a proof, see e.g. \cite[Lemma 16]{Kiz18}.
\begin{lemma}\label{spec_inclusion}
	Let $E$ be a Fréchet space, and let $T\colon E \to E$ be a continuous linear operator with the dual $T'\colon E' \to E'$. Then 
	\[\sigma_{pt}(T';E') \subset \sigma(T;E).\]
\end{lemma}
For each $r>0$ we use the notation $D(r):=\{\lambda \in \cn: |\lambda-\frac{1}{2r}|<\frac{1}{2r}\}$. Let $\alpha:=a(\lambda)$. Then, $|\lambda-\frac{1}{2r}|=\frac{1}{2r}$ if and only if $\alpha=r$.
\begin{proposition}\label{spectrum inclusion general case}
	Let $A$ be a Köthe matrix satisfying (G$\infty$-1). Then, 
	\[\Sigma \subseteq \spec \subseteq \overline{D(1)}.\]
\end{proposition}

\begin{proof}
	Let $\lambda \in \Sigma$, that is, there exists $s \in \nn$ such that $\lambda=\frac{1}{s}$. Define $u^{(s)}$ by 
\[
u^{(s)}_i:=\prod_{j=1}^{i-1}\left(1-\frac{1}{\lambda j}\right),
\]
for $1<i\leq s$ (with $u^{(s)}_1:=1$) and $u^{(s)}_i:=0$ for $i>s$. It is straightforward to show that $u^{(s)} \in \kozero'$ (since $u^{(s)}$ belongs to the space of finitely supported sequences $c_{00}$) and $\ces'u^{(s)}=\frac{1}{s}u^{(s)}$, that is, $\lambda \in \sigma_{pt}(\ces',\kozero')$. By Lemma~\ref{spec_inclusion}, $\lambda \in \spec$. This shows $\Sigma \subseteq \spec$. By \cite[Lemma 2.8]{ABR18-2} we see that $\sigma(\ces_n;c_0(v_n)) \subseteq \overline{D(1)}$ for all $n \in \nn$, for which $\ces_n\colon c_0(v_n) \to c_0(v_n)$ is the restriction of $\ces \in \clo(\cnn)$. Hence $\bigcap_{s \in \nn}\left(\bigcup_{j=s}^\infty \sigma(\ces_j;c_0(v_j))\right)\subseteq \overline{D(1)}$,
and so $\spec \subseteq \overline{D(1)}$ by \cite[Lemma 5.5]{ABR18-2}.
\end{proof}

\begin{proposition}\label{spectrum inclusion non-nuclear case}
		Let $\kotinf$ be a non-nuclear, Schwartz $\ginf$-space. Then
	\[
		\{0,1\} \cup D(1) \subseteq \spec \subseteq \overline{D(1)}.
	\]
\end{proposition}

\begin{proof}
	By Proposition~\ref{nonnuclear consequence} and Proposition~\ref{spectrum inclusion general case} we already know that $\Sigma_0 \subseteq \spec \subseteq \overline{D(1)}$. So it remains to establish $D(1)\setminus \Sigma \subseteq \spec$. Let $\lambda \in D(1)\setminus \Sigma$ and suppose that $\lambda \notin \spec$. Then the inverse operator $(\ces-\lambda I)^{-1}$ is continuous, equivalently, for all $n \in \nn$ there exists $m>n$ such that $(\ces-\lambda I)^{-1}\colon c_0(v_n) \to c_0(v_m)$ is continuous. Let $\beta:=a(\lambda)$ as in \eqref{function a}. Retaining the notation of Proposition~\ref{spectrum nuclear case} it follows that the linear map $\tilde{E}_{\lambda,n,m}\colon c_0 \to c_0$ is continuous, where $\tilde{E}_{\lambda,n,m}=(\tilde{e}_{ij}^{nm}(\lambda))_{i,j}$ is determined by the lower triangular matrix
	\begin{equation}
		\tilde{e}_{ij}^{nm}(\lambda)=\frac{v_m(i)}{v_n(j)}e_{ij}(\lambda), \quad \forall i \geq 2, \, 1\leq j <i,
	\end{equation}
	and $\tilde{e}_{ij}^{nm}(\lambda)=0$, if $j \geq i$. Indeed, 
	\[
	e_{ij}(\lambda)= \frac{1}{i\prod_{k=j}^i \left(1-\frac{1}{\lambda k}\right)}, \quad 1 \leq j <i,
	\]
	and $e_{ij}(\lambda)=0$, if $j \geq i$. Since $\tilde{E}_{\lambda,n,m} \in \clo(c_0)$, by the well-known criterion \cite[Theorem 4.51-C]{Tay58} we necessarily have $\supi\sumj \frac{v_m(i)}{v_n(j)}|e_{ij}(\lambda)| < \infty$. By \cite[pp.776]{ABRwc0} and (G$\infty$-1), there exists $C>0$ such that
	\[
	\supi \sumj \frac{v_m(i)}{v_n(j)}|e_{ij}(\lambda)|\geq C\supi i^{\beta-1}v_m(i) \sumjim\frac{1}{j^\beta v_n(j)} \geq C\supi i^{\beta-1}\frac{v_m(i)}{v_n(1)} \sumjim\frac{1}{j^\beta}.
	\]
	Since $\beta>1$, we have
	\[
	\sumjim \frac{1}{j^\beta} \geq \sumjim\int_j^{j+1} \frac{1}{x^\beta}\mathrm{d} x =\int_1^i \frac{1}{x^\beta}\mathrm{d} x =\frac{1}{\beta-1}\left(1-\frac{1}{i^{\beta-1}}\right).
	\]
	We have shown that for all $n \in \nn$ there is $m>n$ such that
\[
\supi i^\beta v_m(i)\left(1-\frac{1}{i^{\beta-1}}\right)\leq (\beta-1)v_n(1)\supi\sumj \frac{v_m(i)}{v_n(j)}|e_{ij}(\lambda|<\infty.
\]
Taking $n=1$, we get $\supi i^\beta v_m(i)<\infty$, for some $m \in \nn$. By Proposition~\ref{nuclear rapidly decreasing}, $\kozero$ is nuclear. This is a contradiction, and $\lambda \notin \rho(\ces;\kozero)$.
\end{proof}

\begin{remark}\label{log does not cover}
	\normalfont Let $\kotinf$ be a $\ginf$-space. Then the condition
	\[
		\exists n \in \nn:\quad \supi \log(i) v_n(i)<\infty \tag{L}\label{lceqn}
 	\]
	cannot be a nuclearity criterion for $\kozero$. Let $\alpha_i:=\log(\log(i))$, for $i\geq 3^3$ and consider the associated power series space $\pssi$ of infinite type \cite[Remark 3.5-(ii)]{ABR18-2}. By \cite[Proposition 29.6]{Vog97}, $\Lambda^1_\infty$ is nuclear if and only if $\supi \alpha_i^{-1}\log(i)<\infty$. However, we directly observe that $\supi \frac{\log(i)}{\alpha_i} = \infty$. So, $\pssi$ is not nuclear. It is easy to check $\pssi$ satisfies \eqref{condition S} so it is Schwartz. Let $a_n(i):=\exp(\log(\log(i))n)=\log(i)^n$. Then, $\pssi$ is isomorphic to the non-nuclear $\ginf$-space $\kotinf$. For a fixed $n \in \nn$, it is easy to see that $\supi\log(i)v_n(i)<\infty$. Therefore $\kozero$ satisfies condition \eqref{lceqn}.
\end{remark}

\begin{proposition}\label{spectrum non-nuclear case}
	Let $\kotinf$ be a non-nuclear, Schwartz $\ginf$-space.
	\begin{enumerate}[label=\normalfont(\arabic*)]
		\item If $\kozero$ satisfies condition \eqref{lceqn}, then
		 $\spec=\{0,1\} \cup D(1)$.
		 \item If $\kozero$ fails condition \eqref{lceqn}, then
		 $\spec=\overline{D(1)}$.
	\end{enumerate} 
\end{proposition}

\begin{proof}
	Retaining the notation of the proof of Proposition~\ref{spectrum nuclear case}, for each $\lambda \in \cn\setminus\Sigma_0$, $(\ces-\lambda I)^{-1} \in \clo(\cnn)$ satisfies $(\ces-\lambda I)^{-1}=D_\lambda-\frac{1}{\lambda^2}E_\lambda$. In the previous section we have seen that the diagonal in $D_\lambda$ is a bounded sequence, independent of nuclearity condition. So $(\ces-\lambda I)^{-1}\colon \kozero \to \kozero$ is continuous if and only if $E_\lambda \in \clo(\kozero)$. Since $\kozero=\ind_n c_0(v_n)$, $E_\lambda \in \clo(\kozero)$ if and only if for each $n \in \nn$ there exists $m>n$ such that $E_\lambda:c_0(v_n)\to c_0(v_m)$ is continuous. With $\tilde{E}_{\lambda,n,m}=(\tilde{e}_{ij}^{nm})_{i,j}$, where $\tilde{e}_{ij}^{nm}=\frac{v_m(i)}{v_n(j)}e_{nm}(\lambda)$ for $i,j \in \nn$, it follows by the argument used in (ii) of the proof of Proposition~\ref{spectrum nuclear case} that $E_\lambda\colon c_0(v_n) \to c_0(v_m)$ is continuous if and only if $\tilde{E}_{\lambda,n,m}\colon c_0 \to c_0$ is continuous. By \cite[Theorem 4.51-C]{Tay58} it suffices to show that both \eqref{limit 0 condition non-nuclear} and \eqref{summable condition non-nuclear} are satisfied:
\begin{equation}\label{limit 0 condition non-nuclear}
	\limii |\tilde{e}_{ij}^{nm}(\lambda)|=\limii \frac{v_m(i)}{v_n(j)}|e_{ij}^{nm}(\lambda)|=0, \quad \forall j \in \nn
\end{equation}
\begin{equation}\label{summable condition non-nuclear}
	\supi \sumj \frac{v_m(i)}{v_n(j)}|e_{ij}(\lambda)|=\supi\sumjim\frac{v_m(i)}{v_n(j)}|e_{ij}(\lambda)|<\infty.
\end{equation}
If $\lambda \notin\{0,1\}$ belongs to the boundary $\partial D(1)$ of $D(1)$, then $\beta:=a(\lambda)=1$ and $\lambda \notin \Sigma_0$. By \cite[Lemma 3.3]{ABRwc0} there exist $\nu,\gamma>0$ such that
\begin{equation}\label{bound eijlambda}
	\frac{\nu}{j} \leq |e_{ij}(\lambda)| \leq \frac{\gamma}{j}, \quad \forall i \in \nn, \, \, 2 \leq j <i.
\end{equation}
Since $\kotinf$ is Schwartz, for any $n \in \nn$, we find $m>n$ such that,
\[
	\frac{v_m(i)}{v_n(j)}|e_{ij}(\lambda)| \leq \frac{v_m(i)}{v_n(i)	} \frac{\gamma}{j} \igoes 0.
\]
So \eqref{limit 0 condition non-nuclear} is satisfied. Let us recall the well-known inequality
\begin{equation}\label{1/j bounding}
\log(i) \leq \sumjim \frac{1}{j} \leq 1+\log(i-1).
\end{equation}
\begin{enumerate}[wide, label=(\arabic*), labelwidth=!, labelindent=0pt]
\setlength\itemsep{0.5em}
\item Let us assume that there exists $n \in \nn$  as in \eqref{lceqn}. We apply \eqref{bound eijlambda}, (G$\infty$-1), \eqref{1/j bounding} respectively, and then choose $m>n$ and $C>0$ as in (G$\infty$-2) to observe
\begin{align*}
\supi\sumjim \frac{v_m(i)}{v_n(j)}|e_{ij}(\lambda)|  &\leq \gamma\supi  v_m(i)\sumjim \frac{1}{jv_n(j)}  \leq \gamma \supi \frac{v_m(i)}{v_n(i)}\sumjim \frac{1}{j} \\
& \leq \gamma (1+\log(i))\frac{v_m(i)}{v_n(i)}\\
& \leq C \gamma \supi (1+\log(i))v_n(i) < \infty.
\end{align*}
This implies \eqref{summable condition non-nuclear} is satisfied for $\lambda \in \partial(D)\setminus\{0,1\}$, hence $\lambda \in \rho(\ces,\kozero)$. Therefore, by Proposition~\ref{spectrum inclusion non-nuclear case}, $\spec=\{0,1\}\cup D(1)$.

\item Let us apply \eqref{bound eijlambda}, (G$\infty$-1), and \eqref{1/j bounding} respectively to obtain
\begin{align*}
\supi\sumjim \frac{v_m(i)}{v_n(j)}|e_{ij}(\lambda)| &\geq \nu\supi v_m(i) \sumjim \frac{1}{jv_n(j)}  \geq \nu\supi \frac{v_m(i)}{v_n(1)}\sumjim\frac{1}{j}\\
&\geq \nu\supi\frac{\log(i)}{v_n(1)}v_m(i).
\end{align*}
However, $\supi\log(i)v_m(i)=\infty$, by assumption. This means \eqref{summable condition non-nuclear} cannot be satisfied. Hence no $\lambda\in \partial D(1)\setminus \{0,1\}$ exists which satisfies $\lambda \in \rho(\ces;k_0(V))$, that is, $\partial D(1) \setminus \{0,1\} \subseteq \spec$. By Proposition~\ref{spectrum inclusion non-nuclear case}, we are done.
\end{enumerate}
\end{proof}

\section{Mean ergodicity of $\ces$}
Let $X$ be a Fréchet space equipped with the increasing system of seminorms $(p_n(\cdot))_{n \in \nn}$. For $S \in \clo(X)$, the strong operator topology $\tau_s$ in $\clo(X)$ is determined by the seminorms $p_{n,x}(S):= p_n(Sx)$, for all $x \in X$ and for all $n \in \nn$. In this case we write $\clo_s(X)$. Let $\mathcal{B}(X)$ be the family of bounded subsets of $X$. Then, the uniform topology $\tau_b$ in $\clo(X)$ is defined by the family of seminorms $p_{n,B}(S):=\sup_{x \in B} p_n(Sx)$, for all $n \in \nn$ and for all $B \in \mathcal{B}(X)$, where $S \in \clo(X)$. In this case we write $\clo_b(X)$. A Fréchet space operator $T \in \clo(X)$ is called \textit{power bounded} if $(T^k)_{k=1}^\infty$ is an equicontinuous subset of $\clo(X)$. Given $T \in \clo(X)$, the averages $T_{[k]}:=\frac{1}{k} \sum_{j=1}^k T^j$, for  $k \in \nn$ are called the \textit{Cesàro means} of $T$. The operator $T$ is said to be \textit{mean ergodic} (resp., \textit{uniformly mean ergodic}) if $(T_{[k]})_k$ is a convergent sequence in $\clo_s(X)$ (resp., in $\clo_b(X)$). 
\begin{proposition}\label{ces is mean ergodic}
	Let $\kotinf$ be a $\ginf$-space which is Schwartz. Then, the Cesàro operator $\ces \in \clo(\kozero)$ is power bounded and uniformly mean ergodic. In particular,
	\begin{equation}\label{4.1}
		\kozero=\operatorname{ker}(I-\ces)\oplus \overline{(I-\ces)(\kozero)}
	\end{equation}
	with $\operatorname{ker}(I-\ces)=\{\boldsymbol{1}\}$ and
	\begin{equation}\label{density of cesm}
		\overline{(I-\ces)(\kozero)} = \{x \in \kozero:x_1=0\}=\overline{\operatorname{span}\{e_i\}_{i \geq 2}}.
	\end{equation}
\end{proposition}

\begin{proof}
Let $\ces^k$ be the $k$-th iterate of $\ces$. By (G$\infty$-1), $v_n(i+1) \leq v_n(i)$, for all $i \in \nn$. Then, by \cite[Corollary 2.3(i)]{ABRwc0} $\ces \in \clo(c_0(v_n))$ and $q_n(\ces x)\leq q_n(x)$, for all $x \in c_0(v_n)$. By \eqref{continuity by norm}, for every $n \in \nn$ we have $q_n(\ces^k x) \leq q_n(x)$, for all $x \in c_0(v_n)$, and $k \in \nn$. \cite[Lemma 5.4]{ABR18-2} implies $\ces$ is power bounded in $\kozero$. It follows by Propositions 2.4 and 2.8 in \cite{ABR09} that $\ces$ is uniformly mean ergodic in $\kozero$ and hence \eqref{4.1} is also satisfied.
\end{proof}

\begin{proposition}
	Let $\kotinf$ be a Schwartz $\ginf$-space satisfying condition \eqref{lceqn}. Then the range $(I-\ces)^j(\kozero)$ is a closed subspace of $\kozero$ for each $j \in \nn$.
\end{proposition}

\begin{proof}
	First we consider the case $j=1$. Set $X(V):=\{x \in \kozero\colon x_1=0\}$. We claim that $(I-\ces)(\kozero) = (I-\ces)(X(V))$. We proceed as in the proof of the analogous result in \cite[Proposition 4.3]{ABR18-2}. By condition \eqref{tends to infinity}, there exists $n_0 \in \nn$ such that $v_n(i) \igoes 0$, for all $n \geq n_0$. Since $\kozero$ is an inductive limit of increasing Banach spaces, we can clearly assume that $v_n(i) \igoes 0$, for all $n \in \nn$. So each $v_k$ is strictly positive and decreasing with $v_n \in c_0$ and hence $\overline{(I-\ces)(c_0(v_n))}=\{x \in c_0(v_n)\colon x_1=0\}=:X_n$ and $(I-\ces)(X_k)=(I-\ces)(c_0(v_n))$ by \cite[Lemma 4.1 and Lemma 4.5]{ABRwc0}. If $x \in X(V)$, then $x \in X_n$ for some $n \in \nn$ and $(I-\ces)x \in (I-\ces)(X_n)=(I-\ces)(c_0(v_n)) \subseteq (I-\ces)(\kozero)$. That fulfills one inclusion. Now let $x \in \kozero$. Then $x \in c_0(v_n)$ for some $n \in \nn$ and hence $(I-\ces)x \in (I-\ces)(c_0(v_n))=(I-\ces)(X_n) \subseteq (I-\ces)(X(V))$. Hence $(I-\ces)(\kozero) = (I-\ces)(X(V))$. To prove that $(I-\ces)(\kozero)$ is closed in $\kozero$, it suffices to show that $(I-\ces) \in \clo(X(V))$ is surjective: if $(I-\ces)(X(V))=X(V)$, then $(I-\ces)(\kozero)=X(V)$ and hence $(I-\ces)(\kozero)$ is closed in $\kozero$. By \cite[Lemma 6.3.1]{PCB87}, $(X(V),\tau)=\ind_n X_n$, where $\tau$ is the relative topology in $X(V)$ induced from $\kozero$. If we set $\tilde{v}_n(i):=v_n(i+1)$, for all $i,n \in \nn$, then we have the topological isomorphism $X(V) \simeq E:=\ind_n c_0(\tilde{v}_n)$ by the left shift operator $S \colon X(V) \to E$ which is a surjective isomorphism at each step $S \colon X_n \to c_0(v_n)$. Let $T:=S \circ (I-\ces)|_{X(V)} \circ S^{-1} \in \clo(E)$. We now prove that $A$ is bijective with $B:=T^{-1} \in \clo(E)$. It is straightforward to see $T\colon \cnn \to \cnn$ is bijective and its inverse $B$ is given by the lower triangular matrix $(b_{ij})$ whose entries are $\frac{1}{j}$ for $1 \leq j<i$, $\frac{i+1}{i}$ for $j=i$, and $0$ for $j>i$. To show that $B$ is still the inverse of $T$ when acting on $E$, we must prove $B \in \clo(E)$, equivalently, for each $n \in \nn$ there exists $m>n$ such that $\phi_m \circ B \circ \phi_n^{-1} \in \clo(c_0)$, where the surjective isometry $\phi_k\colon c_0(\tilde{v}_k)\to c_0$ is given by $\phi_k x=(v_k(i+1)x_i)$ for every $x \in c_0(\tilde{v}_k)$. The lower triangular matrix corresponding to $\phi_m \circ B \circ \phi_n^{-1}$ is given by $d_{ij}:=(\frac{v_m(i+1)}{v_n(j+1)}b_{ij})$, for all $i,j \in \nn$. For every $j$, $\limii \frac{v_m(i+1)}{v_n(j+1)}b_{ij}=\frac{1}{jv_n(j+1)}\limii v_m(i+1) =0$. We make use of (G$\infty$-1) and \eqref{1/j bounding}, respectively, then pick $m>n$ and $C>0$ as in (G$\infty$-2), and then finally use condition \eqref{lceqn} to observe
\begin{align*}
\sumj \frac{v_m(i+1)}{v_n(j+1)}b_{ij} &= \frac{i+1}{i}\frac{v_m(i+1)}{v_n(i+1)}\sumjim \frac{1}{jv_n(j+1)} \leq \frac{v_m(i+1)}{v_n(i+1)}\left(\frac{i+1}{i}+\sumjim \frac{1}{j}\right)\\
& \leq \frac{v_m(i+1)}{v_n(i+1)}\left(\frac{i+1}{i}+(1+\log(i-1)) \right) \\
& \leq  C(3+\log(i))v_n(i+1) < \infty.
\end{align*}
Hence both the conditions (i) and (ii) of \cite[Lemma 2.1]{ABRwc0} hold. Then, $\phi_m \circ B \circ \phi_n^{-1} \in \clo(c_0)$ and hence $(I-\ces)(\kozero)$ is closed. Since $(I-\ces)(\kozero)$ is closed, \eqref{4.1} implies $\kozero=\operatorname{ker}((I-\ces)) \oplus (I-\ces)(\kozero)$. The proof of (2) $\Rightarrow$ (5) in \cite[Remark 3.6]{ABR13} implies that $(I-\ces)^j(\kozero)$ is closed in $\kozero$, for all $j \in \nn$.
\end{proof}
Let $X$ be a separable Fréchet space. Then the operator $T \in \clo(X)$ is called \textit{hypercyclic} if there exists $x \in X$ such that the orbit $\{T^kx:k \in \nn_0\}$ is dense in $X$. If, for some $z \in X$, the projective orbit $\{\lambda T^kz:\lambda \in \cn, k \in \nn_0\}$ is dense in $X$, then $T$ is called \textit{supercyclic}. Clearly, if $\ces$ is hypercyclic then $\ces$ is supercyclic.

\begin{proposition}
	Let $A$ be a Köthe matrix. Then $\ces \in \clo(\kozero)$ is not supercyclic, and hence not hypercyclic either.
\end{proposition}

\begin{proof}
	Follows from \cite[Proposition 4.3]{ABR17} and \cite[Proposition 4.4]{ABR18-2}.
\end{proof}
\section*{Acknowledgements}
The author wishes to thank Prof. José Bonet for useful suggestions and discussions.

\bibliographystyle{plain}
\bibliography{cesaro}
\end{document}